\documentclass[12pt]{amsart}
\usepackage{amsmath,amsthm,amsfonts,amssymb,eucal}

\renewcommand {\a}{ \alpha }

\newcommand{\y}{\eta}

\newcommand{\vare}{\varepsilon}

\newcommand{\G}{\Gamma}

\newcommand{\varf}{\varphi}
\renewcommand{\d}{\delta}

\newcommand{\D}{\Delta}
\newcommand{\s}{\sigma}
\newcommand{\Sg}{\Sigma}
\renewcommand{\l}{\lambda}

\renewcommand{\t}{\tau}

\newcommand{\om}{\omega}
\newcommand{\Om}{\Omega}
\newcommand{\x}{\xi}
\newcommand{\R}{ \mathbb R}

\newcommand{\N}{ \mathbb N}
\newcommand{\Tor}{ \mathbb T}
\newcommand{\Z}{ \mathbb Z}

\newcommand{\CL}{\mathcal L}

\newcommand{\CC}{\mathcal C}

\newcommand{\CF}{\mathcal F}

\newcommand{\CH}{\mathcal H}
\newcommand{\CI}{\mathcal I}
\newcommand{\CJ}{\mathcal J}

\newcommand{\CP}{\mathcal P}

\newcommand{\CN}{\mathcal N}

\newcommand {\GH}{\mathfrak H}

\newcommand {\GS}{\mathfrak S}

\newcommand {\bk}{\mathbf k}

\newcommand {\bb}{\mathbf b}

\newcommand {\BB}{\mathbf B}

\newcommand {\BF}{\mathbf F}

\newcommand {\BH}{\mathbf H}
\newcommand {\BI}{\mathbf I}

\newcommand {\BM}{\mathbf M}
\newcommand {\BN}{\mathbf N}

\newcommand {\BS}{\mathbf S}
\newcommand {\BT}{\mathbf T}
\newcommand {\BV}{\mathbf V}
\newcommand {\BW}{\mathbf W}

\newcommand{\wt}{\widetilde}
\newcommand{\wh}{\widehat}

 \DeclareMathOperator{\dist}{dist}

\DeclareMathOperator{\ess}{ess}

\DeclareMathOperator{\rank}{rank}

\newtheorem{thm}{Theorem}[section]
\newtheorem{cor}[thm]{Corollary}
\newtheorem{lem}[thm]{Lemma}
\newtheorem{prop}[thm]{Proposition}

\theoremstyle{definition}

\theoremstyle{remark}

\numberwithin{equation}{section}

\newcommand{\thmref}[1]{Theorem~\ref{#1}}

\newcommand{\bsymb}{\boldsymbol}

\newcommand{\sh}{Schr\"odinger }
\newcommand{\vs}{\vskip0.2cm}

\begin{document}

 \title[Schr\"odinger operator on $\Z^d$]{On the spectral estimates for the
Schr\"odinger operator on $\Z^d,\ d\ge3$}
\author[Rozenblum]{Grigori Rozenblum}
\address[G. Rozenblum]{1) Department of Mathematical Sciences,                       Chalmers University of Technology, Sweden; 2) Department of Mathematical Sciences, University of Gothenburg, Sweden}
\email{grigori@chalmers.se}
\author[Solomyak]{Michael Solomyak}
\address[M. Solomyak]{Department of Mathematics, The Weizmann Institute of Science, Rehovot, Israel}
\email{michail.solomyak@weizmann.ac.il}
\begin{abstract}For the discrete Schr\"odinger operator we obtain sharp estimates for the number of negative eigenvalues.
\end{abstract}
\maketitle
\section{Introduction}\label{intro}
We study the estimates of the number of negative eigenvalues of
the discrete Schr\"odinger operator
\begin{equation}\label{0:1}
    \BH_{\a V}=-\D-\a V
\end{equation}
in the Hilbert space $\ell_2(\Z^d),\ d\ge 3$. We use the standard
notation for the lattice points: $x=(x_1,\ldots,x_d)$, with
$x_j\in\Z$ for each $j$. The discrete Laplacian is
\begin{equation*}
    (\D u)(x)=\sum_{j=1}^d(u(x+1_j)+u(x-1_j)-2u(x)), \qquad x\in \Z^d,
\end{equation*}
 where $1_j$ is the multi-index with all zero entries except $1$ in the
position $j$. The Laplacian $\D$ is a bounded operator, and the
spectrum of $-\D$ is absolutely continuous and coincides with
$[0,2d]$. The corresponding quadratic form is
\begin{equation}\label{0:3}
    Q_0[u]=\sum_{x\in\Z^d}\sum_{j=1}^d|u(x+1_j)-u(x)|^2.
\end{equation}
Being considered on the set $\CF$ of sequences $u$ with finite
support, the quadratic form $Q_0$ is non-degenerate. For $u\in\CF$
the Hardy inequality is fulfilled,
\begin{equation}\label{0:Hardy} Q_0[u]\ge
c_d\sum_{x\in\Z^d}\frac{|u(x)|^2}{|x|^2+1},\qquad d\ge
3;\end{equation} see Section \ref{hardy1} where we discuss this
inequality and its generalizations. It follows from
\eqref{0:Hardy} that the completion of $\CF$ in the metric
generated by the quadratic form $Q_0$ is some Hilbert space of
number sequences. We denote it by $\CH^1=\CH^1(\Z^d)$. The symbol
$V$ in \eqref{0:1} stands for the discrete potential $V=V(x)\ge
0$. We usually assume that $V(x)\to 0$ as $|x|\to\infty$; then the
operator of multiplication by $V$ is compact in $\ell_2(\Z^d)$.
Finally, $\a\ge 0$ in \eqref{0:1} is a large parameter (the
coupling constant). \vs
 If $V\to 0$ at infinity
then for any $\a>0$ the essential spectrum $\s_{\ess}(\BH_{\a V})$
is the same as for $\a=0$, i.e., it is $[0,2d]$. The negative
spectrum consists of a finite or countable set of eigenvalues,
each of a finite multiplicity, with the only possible accumulation
point at $\l=0$. We denote by $N_-(\BH_{\a V})$ the number of
negative eigenvalues, counted with their multiplicities. If $V$
has infinite support, then $N_-(\BH_{\a V})\to\infty$ as $\a$
grows. For finitely supported $V$ one obviously has
\begin{equation}\label{0:5}
    N_-(\BH_{\a V})\le \#\{x\in\Z^d: V(x)\neq 0\},\qquad \forall \a>0.
\end{equation}
Our main goal is to find order-sharp estimates for $N_-(\BH_{\a
V})$, depending on the properties of the potential $V$ and on the
value of $\a$. \vs

A similar problem for the `continuous' Hamiltonian is well
studied, see e.g. the survey paper \cite{RS} and references therein.
Here we recall the basic results for the latter problem, since the
comparison of these two problems will be one of our main
concerns.\vs

The most important result for the continuous case is given by the
Rozenblum -- Lieb -- Cwikel inequality. We present its most
complete formulation, see Theorem 2.1 in \cite{RS}. This
formulation goes back to the lectures \cite{BSlect}, see Theorems
4.14, 4.15, and 4.17 there. For the Schr\"odinger operator on
$\R^d$ we use the same notation \eqref{0:1} as for its discrete
counterpart; the meaning of all terms in the formulas
\eqref{1:rlc} and \eqref{0:W} below should be clear from the
context.
\begin{thm}\label{RLC} Let $d\ge 3$. Then there exists a
constant $C_{1.5}=C_{1.5}(d)$ such that for any $ V\in
L_\frac{d}2(\R^d),\ V\ge0$,
\begin{equation}\label{1:rlc}
    N_-(\BH_{\a V})\le C_{1.5}\a^\frac{d}2\int_{\R^d}V^\frac{d}2dx,
\end{equation}
and moreover, the Weyl asymptotic formula holds:
\begin{equation}\label{0:W}
    N_-(\BH_{\a V})\sim
    \omega_d\a^\frac{d}2\int_{\R^d}V^\frac{d}2dx,\qquad\a\to\infty.
\end{equation}
Conversely, suppose that $d\ge3$, for a certain $V\ge 0$ the
operator $\BH_{\a V}$ is well defined (via its quadratic form) and
bounded below for all $\a>0$, and $N_-(\BH_{\a
V})=O(\a^\frac{d}2)$ as $\a\to\infty$. Then $V\in
L_\frac{d}2(\R^d)$, and, therefore, estimate \eqref{1:rlc} and
asymptotic formula \eqref{0:W} are fulfilled.
\end{thm}
\vs We immediately conclude from \thmref{RLC} that in the
continuous case the behavior $N_-(\BH_{\a V})=o(\a^\frac{d}2)$ is
impossible, unless $V\equiv0$. The growth faster than $N_-(\BH_{\a
V})=O(\a^\frac{d}2)$ is possible, see the results in \cite{BS,
BS2}; some of them are also presented in \cite{RS}. \vs In the
discrete case an analogue of \eqref{1:rlc} remains valid, but
other statements of \thmref{RLC} fail to be true. Let us formulate
the result.
\begin{thm}\label{RLC-d} Let $d\ge 3$. Then there exists a
constant $C_{1.7}=C_{1.7}(d)$ such that for any $ V\in
\ell_\frac{d}2(\Z^d),\ V\ge0,$
\begin{equation}\label{1:rlc-d}
    N_-(\BH_{\a V})\le
    C_{1.7}\a^\frac{d}2\sum_{x\in\Z^d}V(x)^\frac{d}2.
\end{equation}
At the same time,
\begin{equation}\label{0:W-d}
    N_-(\BH_{\a V})=o(\a^\frac{d}2).
   \end{equation}
\end{thm}
The estimate \eqref{1:rlc-d} is known, see e.g. \cite{LS,RS}. The
property \eqref{0:W-d} is quite elementary, but, probably, was
observed for the first time in the survey paper \cite{RS}. See
Section \ref{discrRLC} below for a more detailed discussion of
\thmref{RLC-d}.

Comparing these theorems, we readily see not only analogies but
also distinctions between the continuous and the discrete cases.
There are also others facts of a similar nature which we are
going to discuss.\vs

Let us describe the structure of the paper. In the next Section
\ref{prel} we present the necessary auxiliary material, and then
we explain the proof of \thmref{RLC-d}. In Section \ref{2q<d} we
obtain the estimates of the type $N_-(\BH_{\a V})=O(\a^q)$ with
$2q<d$. They have no analogues in the continuous case.

Technically, the case $2q>d$ is more involved than the previous
one. The general multidimensional discrete Hardy type inequalities
are an important ingredient of our approach. They are known in
`mathematical folklore', but we could not find any exposition of
this material in the literature. To fill this gap, we give such an
exposition in Section \ref{hardy1}. { We describe a wide class of
discrete Hardy inequalities that can be derived directly from
their classical analogues for $\R^d$.}

Section \ref{2q>d} is devoted to the estimates $N_-(\BH_{\a
V})=O(\a^q)$ with $2q>d$. Based upon the Hardy inequalities
established in Section \ref{hardy1}, we obtain the upper estimates
similar to those known for the continuous {\sh} operator. We also
analyze an example showing that these estimates are order-sharp.

In Section \ref{sparse} we analyze a special class of
so-called {\it sparse potentials}. The estimates we derive for
such potentials, depend not only on the values of $V(x)$ (as, say,
in  \eqref{1:rlc-d}), but also on the geometry of the
support $\{x: V(x)\neq0\}$. In particular, the results of this
section allow us to construct potentials $V$ such that
\begin{equation*}
    N_-(\BH_{\a V})=O(\a^\frac{d}2)\qquad \text{but}\qquad
    N_-(\BH_{\a V})\neq o(\a^\frac{d}2).
\end{equation*}
By \eqref{0:W-d}, such potentials cannot lie in $\ell_\frac{d}2$.

The results of Section \ref{sparse} show also that the class of
Hardy inequalities, described in Section \ref{hardy1}, does not
cover all the possibilities. More precisely, we prove that any
bounded sparse potential is a discrete Hardy weight, which may be
considered as a rather unexpected result. We consider the material
of Section \ref{sparse} as the most important in the whole paper.

The concluding Section \ref{concl} is devoted to a discussion of some
related results. \vs

Our notation is rather standard, or it is explained in the course
of the presentation. We denote by $C,c$, etc., without an index,
various constants whose value is indifferent for us. Our notation
for the more important constants is clear from the following
example: $C_{1.5}$ is the constant appearing for the first time in
the inequality (1.5). Sometimes, we explicitly indicate the
parameters affecting the value of $C$.\vskip0.3cm

The work on the paper started in April 2008, when G.R. was
enjoying the hospitality of the Weizmann Institute of Science.
G.R. acknowledges his gratitude to the Institute.

We also express our thanks to Professor V. Maz'ya for a useful
discussion.

\section{Preliminaries. Proof of \thmref{RLC-d}}\label{prel}
\subsection{Birman -- Schwinger operator.}\label{B-Sch}
When studying the quantity $N_-(\BH_{\a V})$, it is usually
convenient to move on to another, but equivalent, setting of the
problem.

Let $d\ge3$, and let $V$ be a bounded, real-valued function on
$\Z^d$. In the Hilbert space $\CH^1(\Z^d)$ we consider the
quadratic form
\begin{equation}\label{1:1}
    \bb_V[u]=\sum_{x\in\Z^d}V(x)|u(x)|^2.
\end{equation}
If $\bb_V[u]$ is bounded in $\CH^1(\Z^d)$, it generates in this
space a bounded linear operator, say $\BB_V$. We call it {\it the
Birman -- Schwinger operator} for our original problem. Since $V$
is supposed to be real-valued, the operator $\BB_V$ is
self-adjoint. If $V\ge 0$, it is non-negative.

For a non-negative compact operator $\BT$, we denote by
$\l_j(\BT)$ its positive eigenvalues, counted according to their
multiplicities and numbered in order of decrease. The symbol
$n_+(s,\BT),\ s>0$, stands for the distribution function of the
eigenvalues:
\[n_+(s,\BT)=\#\{j:\l_j(\BT)>s\}.\]

The following statement is a particular case of the general Birman
-- Schwinger principle; see, e.g., \cite{BS2} for its exposition.

\begin{prop}\label{BSchw} Let $V\ge 0$. If the operator $\BB_V$ is
compact, then for the Schr\"odinger operator \eqref{0:1} the
number of negative eigenvalues is finite for any $\a>0$. Moreover,
the equality holds:
\begin{equation}\label{1:2}
    N_-(\BH_{\a V})=n_+(\a^{-1},\BB_V),\qquad\forall \a>0.
\end{equation}

Conversely, if $N_-(\BH_{\a V})<\infty$ for all $\a>0$, then the
operator $\BB_V$ is compact and, therefore, \eqref{1:2} is valid.
\end{prop}
\vs

Recall that the modulus of an operator $\BT$ is defined as
\[ |\BT|=(\BT^*\BT)^\frac12.\]
Evidently, $|\BT|=\BT$ if the latter operator is non-negative. The
compact operators, such that for some $q>0$
\[ n_+(s,|\BT|)=O(s^{-q}),\]
or, equivalently, $\l_j(|\BT|)=O(j^{-\frac1{q}})$, form a linear
space (an ideal in the algebra of all bounded operators). This
space is usually denoted by $\GS_{q,\infty}$, or by $\Sg_q$. The
latter notation is used in the book \cite{BSbook}, and we keep it
here.
 The spaces $\Sg_q$ are the `weak analogues'
of the classical Schatten ideals $\GS_q$. The functional
\begin{equation}\label{1:6}
    \|\BT\|_{\Sg_q}=
    \sup_{s>0}\, s n_+(s,|\BT|)^\frac1{q}
    =\sup_j\,j^\frac1{q}\l_j(|\BT|)
\end{equation}
defines a quasi-norm on $\Sg_q$. If $q>1$, and only in this case,
an equivalent norm exists on $\Sg_q$. However, the quasi-norm
\eqref{1:6} is more convenient for estimates.

The spaces $\Sg_q$ are complete and non-separable. The condition
\[n_+(s,|\BT|)=o(s^{-q}), \  s\to0,\] or, equivalently,
$\l_j(|\BT|)=o(j^{-\frac1{q}})$, singles out a closed separable
subspace of $\Sg_q$, which we denote by $\Sg_q^\circ$. See Section
11.6 in the book \cite{BSbook} for more detail on the classes
$\Sg_q$ and $\Sg_q^\circ$. \vs

Along with the quasi-norm \eqref{1:6}, let us consider the
functionals
\begin{equation*}
    \D_q(\BT)=\limsup\limits_{s\to 0}(s^qn_+(s,\BT)),\qquad
\d_q(\BT)=\liminf\limits_{s\to 0}(s^qn_+(s,\BT)).
\end{equation*}
They are well-defined for any $\BT\in\Sg_q$. The following
elementary fact, see Theorem 11.6.7 in \cite{BSbook}, will be
useful later.
\begin{prop}\label{lims} The functionals $\D_q,\d_q$ are continuous
in the topology of the space $\Sg_q$.
\end{prop} \vs

Proposition \ref{BSchw} allows one to replace the study of the
function $N_-(\BH_{\a V})$ with the investigation of the compact
operator $\BB_V$. This is convenient, since it makes it possible
to use the powerful machinery of the theory of compact operators.

For instance, for any $q>0$ the powerlike estimate
\begin{equation*}
    N_-(\BH_{\a V})\le C\a^q,\qquad \forall\a>0
\end{equation*}
is equivalent to $\BB_V\in\Sg_q$, with $\|\BB_V\|_{\Sg_q}\le
C^\frac1q$. In a similar way,
\begin{equation*}
    N_-(\BH_{\a V})=o(\a^q)\ \Longleftrightarrow
 \ \BB_V\in\Sg_q^\circ.
\end{equation*}\vs

In particular, the estimate \eqref{1:rlc-d}, being reformulated in
terms of the operator $\BB_V$, reads in either of two equivalent
forms:
\begin{equation}\label{rlc-norm}
    \|\BB_V\|_{\Sg_\frac{d}2}\le C_{1.7}^\frac2{d}\|V\|_{\ell_\frac{d}2};
    \qquad
    n_+(s,\BB_V)\le C_{1.7}s^{-\frac{d}2}\sum_{x\in\Z^d}V(x)^{\frac{d}2}.
\end{equation}

\vs

The spaces $\Sg_q$ and $\Sg_q^\circ$ have their counterparts
$\ell_{q,w}$ and $\ell^\circ_{q,w}$ in the theory of function
spaces on countable sets. Let a function $V$ on $\Z^d$ be such
that $V(x)\to 0$ as $|x|\to\infty$. By re-arranging the numbers
$|V(x)|,\ x\in\Z^d$, in the non-increasing order, we obtain a
sequence of non-negative numbers (notation $V^*_j,\ j\in\N$). Then
$V^*_j\to 0$ as $j\to\infty$.

Besides, for any $s>0$ we denote
\begin{equation}\label{1:8}
    E(s,V)=\{x\in\Z^d:|V(x)|>s\};\qquad \nu(s,V)=\#E(s,V).
\end{equation}
 By definition,
\[ \{V\in\ell_{q,w}\}\ \Longleftrightarrow\
\{\nu(s,V)=O(s^{-q})\}\ \Longleftrightarrow\
\{V^*_j=O(j^{-\frac1{q}})\},\] and
\[ \{V\in\ell_{q,w}^\circ\}\ \Longleftrightarrow
\{\nu(s,V)=o(s^{-q})\}\ \Longleftrightarrow\
\{V^*_j=o(j^{-\frac1{q}})\}.\]
 The functional
\begin{equation*}
    \|V\|_{\ell_{q,w}}=
\sup_{s>0}\,s\nu(s,V)^\frac1{q} =\sup_j\,j^{\frac1{q}}V^*_j
\end{equation*}
defines the standard quasi-norm in $\ell_{q,w}$. \vs

The next statement elucidates the role of classes $\ell_{q,w}$ and
$\ell_{q,w}^\circ$ in the problems we are studying. For its proof,
see Section \ref{hardy} below.
\begin{prop}\label{role}
Let $d\ge 3$ and $V\in\ell_{\frac{d}2,w}(\Z^d),\ V\ge0.$ Then the
operator $\BB_V$ is bounded. If
$V\in\ell_{\frac{d}2,w}^\circ(\Z^d)$, this operator is compact.
\end{prop}
Note that the conditions given by this proposition are only
sufficient, but not necessary. Say, they are violated for the
sparse potentials considered in Section \ref{sparse}.

\subsection{On \thmref{RLC-d}.}\label{discrRLC}
As it was already mentioned in Section \ref{intro}, \thmref{RLC-d}
is basically known. Still, in order to make our exposition
self-contained, we outline its proof.

{\it Proof of \thmref{RLC-d}.} It is well known that the discrete
Laplacian generates a positivity preserving semigroup in the space
$\ell_2(\Z^d)$. If $d\ge3$, the following lower estimate, with
some $c>0$, is satisfied for the quadratic form $Q_0$ in
\eqref{0:3}:
\begin{equation*}
    Q_0[u]\ge c\left(\sum_{x\in\Z^d}|u(x)|^\sigma\right)^\frac2{\sigma},\
    \sigma=\frac{2d}{d-2},\qquad\forall u\in \CH^1(\Z^d).
\end{equation*}
This is a particular case of Theorem IV.5.2 in \cite{VCS}. Thus,
the assumptions of Theorem 1.2 in \cite{LS} are fulfilled, whence
the estimate \eqref{1:rlc-d}.

Another way to obtain the latter estimate is to derive it from the
results of the paper \cite{RS1}; see also Section 9 in \cite{RS}.
It is also possible  to deduce \eqref{1:rlc-d} directly from
\eqref{1:rlc}. This will be explained in Section \ref{eigest}.\vs

For justifying \eqref{0:W-d}, let us consider the operator
$\BB_V$. The sequences $V$ with finite support form a dense subset
of $\ell_\frac{d}2(\Z^d)$. By \eqref{0:5}, for such $V$ the
non-zero spectrum of $\BB_V$ is finite and therefore,
\eqref{0:W-d} is satisfied, so that $\D_{\frac{d}2}(\BB_V)=0$.
Using the inequality \eqref{rlc-norm} and Proposition \ref{lims},
we conclude that \eqref{0:W-d} extends to all non-negative
potentials $V\in \ell_{\frac{d}2}(\Z^d)$.

The proof of \thmref{RLC-d} is complete.\vs

Note that the estimate \eqref{0:W-d} does not replace the,
formally weaker, estimate \eqref{1:rlc-d}. An important property
of \eqref{1:rlc-d} is the explicit dependence on the function $V$.
It is this dependence that enabled us, by approximating $V$ by
potentials with finite support, to pass to the limit and obtain
\eqref{0:W-d}.

\section{Estimates in $\Sg_q,\ q<\frac{d}2$}\label{2q<d}
As it is clear from the comparison of Theorems \ref{RLC} and
\ref{RLC-d}, the estimates obtained in this section have no
analogues for the operators on $\R^d$. We are going to show that,
unlike in the continuous case, the function $N_-(\BH_{\a V})$ can
grow as $O(\a^q)$ with arbitrarily small $q>0$, which corresponds
to the arbitrarily fast powerlike decay of the eigenvalues
$\l_j(\BB_V)$.

\begin{thm}\label{4:Thm}Let $d\ge3,\ q< \frac{d}{2}$, and
$V\in \ell_{q,w}(\Z^d),\ V\ge0$. Then $\BB_V\in\Sg_q$, and
\begin{equation}\label{4:Thm.2}
     \|\BB_V\|_{\Sg_q}\le C_{3.1}\|V\|_{\ell_{q,w}},
\end{equation}
with a constant depending only on $d$ and $q$. Equivalently,
\begin{equation*}
    N_-(\BH_{\a V})\le C_{3.1}^q\a^q \|V\|_{\ell_{q,w}}^q.
    \end{equation*}
\end{thm}
\vs For $q=\frac{d}2$ the statement is no more true: the inclusion
$V\in\ell_{\frac{d}2,w}$ guarantees only the boundedness of
$\BB_V$.
\begin{proof}[Proof of Theorem \ref{4:Thm}] For a fixed $s$,
let us split $V$ into two terms,
\begin{equation*}
    V=V_1+V_2,
\end{equation*}
where $V_1(x)=V(x)$ at the vertices $x\in E(s,V)$, cf.
\eqref{1:8}, and $V_1(x)=0$ otherwise. By the Weyl inequality
(see, e.g., Theorem 9.2.9 in \cite{BSbook}),
\begin{equation}\label{4:Thm.3}
     n_+(s, \BB_{ V})\le  n_+({s}/2,\BB_{V_1})+
     n_+({s}/2,\BB_{V_2}).
\end{equation}
The first term in \eqref{4:Thm.3} is estimated by \eqref{0:5},
\begin{equation}\label{first}
    n_+(s/2,\BB_{V_1})\le \nu(s/2,V)\le(2P/s)^q,
\qquad P=\|V\|_{\ell_{q,w}}.
\end{equation}
For the second  term we apply the estimate \eqref{rlc-norm}, with
$V$ replaced by $V_2$. This gives
\begin{equation}\label{second}
    n_+({s}/2,\BB_{V_2})\le 2^{\frac{d}2}C_{1.7}
    s^{-\frac{d}{2}}\sum_{V(x) \le s}V(x)^{\frac{d}2}.
\end{equation}
Now we transform the latter sum:
\begin{gather*}
\sum_{V(x) \le s}V(x)^{\frac{d}2}=-\int_0^s
\t^{\frac{d}2}d\nu(\t,V)=\frac{d}2\int_0^s\t^{\frac{d}2-1}
\nu(\t,V)d\t\\ \le {\frac{d}2}P^q\int_0^s\t^{\frac{d}2-1-q}d\t=
\frac{d}{d-2q}s^{\frac{d}2-q}P^q.
\end{gather*}
Together with \eqref{4:Thm.3}, \eqref{first}, and \eqref{second},
this gives $n(s,\BB_V)\le CP^qs^{-q}$ which is equivalent to
\eqref{4:Thm.2}.
\end{proof}

It is possible also  to obtain the lower estimate of
$n_+(s,\BB_V)$ in terms of the function $\nu(\cdot,V)$. This
estimate does not require any preliminary assumptions about the
non-negative function $V$.
\begin{thm} For any $V\ge0$ the inequality holds:
\begin{equation}\label{4.thmbelow}
    n_+(s, \BB_{ V})\ge 2^{-d} \nu(2sd,V).
\end{equation}
\end{thm}
\begin{proof} For a fixed $\t$, we consider the set $E(\t, V)$.
Let the sublattice $(2\Z)^d$ consist of the points in $\Z^d$ with
all components divisible by $2$. For any $\bk\in \{0,1\}^d$
consider the set $(2\Z)^d+\bk$. These sets are mutually disjoint.
Therefore, the sets $\Om_\bk(V)=E(\t, V)\cap((2\Z)^d+\bk)$ are
disjoint as well, and at least for one value of $\bk\in
\{0,1\}^d$, we have
\begin{equation*}
    \#\Om_\bk(V)\ge 2^{-d}\nu(\t,V).
\end{equation*}
Now, consider the subspace $\CL\subset\CH^1(\Z^d)$ formed by the
functions
\[u(x)=\sum_{y\in\Om_\bk(V)}c_y\d(x-y).\]
The functions $\d(x-y),\ y\in\Om_\bk(V)$, are mutually orthogonal
both in the metric of $\CH^1(\Z^d)$ and with respect to the
quadratic form $\bb_V$ in \eqref{1:1}. So, for any $u\in\CL$ we
have $Q_0[u]=2d\sum_{y\in\Om_\bk(V)}|c_y|^2$, while
$\bb_V[u]=\sum_{y\in\Om_\bk(V)}|c_y|^2 V(y)$. Since $V(y)\ge\t$
for $y\in\Om_\bk(V)$, we have constructed a subspace of dimension
greater than $2^{-d}\nu(\t,V)$ such that $\bb_V[u]\ge (2d)^{-1}\t
Q_0[u]$. This immediately implies \eqref{4.thmbelow} by the
variational principle.
\end{proof}
\begin{cor}\label{twosided}
Let $d\ge3$, $0<q<\frac{d}2$, and $V\ge0$. Then $\BB_V\in \Sg_q$
if and only if $V\in\ell_{q,w}(\Z^d)$, and, moreover, there are
constants $c_0,c_1>0$ such that
\[ c_0\|V\|_{\ell_{q,w}(\Z^d)}\le \|\BB_V\|_{\Sg_q}\le
c_1\|V\|_{\ell_{q,w}(\Z^d)},\] and also
\begin{equation}\label{lsup}
    c_0\limsup_{s\to0}s^q\nu(s,V)\le\D_q(\BB_V)\le c_1
\limsup_{s\to0}s^q\nu(s,V).
\end{equation}
In particular,
\[ N_-(\BH_{\a V})=o(\a^q)\ \Longleftrightarrow\
V\in\ell_{q,w}^\circ(\Z^d).\]
\end{cor}

Here only the relation \eqref{lsup} needs a justification. To this
end, let us denote
\[ R(V)=\limsup_{s\to0}\,s\nu(s,V)^\frac1{q},\]
then $R(V)\le\|V\|_{\ell_{q,w}}$. For any $\vare>0$, one can
change the values of $V(x)$ at a finite number of points
$x\in\Z^d$ in such a way that for the new potential, say
$V_\vare$, we have
\[ \|V_\vare\|_{\ell_{q,w}}\le R(V)+\vare.\]
Then $\D_q(\BB_{V_\vare})=\D_q(\BB_V)$ and, by \eqref{4:Thm.2},
\[\D_q(\BB_V)\le \|\BB_{V_\vare}\|_{\Sg_q}^q
\le C_{3.1}^q(R(V)+\vare)^q.\] Since $\vare$ is arbitrary, we come
to the right inequality in \eqref{lsup}. The left inequality is a
direct consequence of \eqref{4.thmbelow}.

\section{Discrete Hardy inequalities}\label{hardy1}
Here we collect the material on the discrete multidimensional
Hardy type inequalities. We need them for studying the estimates
of the operator $\BB_V$ in the classes $\Sg_q$ with $q>\frac{d}2$.
As it was mentioned in the Introduction, this material should be
considered as known on the `folklore level'. However, we could not
find its exposition in the literature, and decided to present it
here.

{ In this section we describe a rather simple class of discrete
Hardy inequalities. Namely, with any function $W(x)\ge0$ on $\Z^d$
we associate a function $\BW(\x)$ on $\R^d$, which assumes the
value $W(x)$ on the unit cell determined by the vertex $x$; see
Subsection \ref{poly} for details. Then we show that if the
"continuous" Hardy inequality is satisfied with the weight
$\BW(\x)$, then the discrete Hardy inequality with the weight
$W(x)$ holds on $\Z^d$. Naturally, this class of discrete Hardy
inequalities is rather restricted. Later, in Section \ref{sparp},
we will show that there exist discrete Hardy inequalities of a
different origin.}

\subsection{Operator of poly-linear interpolation.}\label{poly}
Let $\CH^1(\R^d), d\ge3,$ stand for the homogeneous Sobolev space.
The metric in $\CH^1(\R^d)$ is defined by the standard Dirichlet
integral,
\[ D[U]=\int_{\R^d}|\nabla U(\x)|^2d\x.\]

We introduce an operator $\CI:\CH^1(\Z^d)\to \CH^1(\R^d)$ in the
following way. In every elementary cubic cell in $\R^d$ we
interpolate $2^d$ values of $u$ at the vertices of the cell by a
poly-linear function. For instance, the function $u(x), \
x\in\{0,1\}^d$, such that $u(0,\ldots,0)=1$ and $u(x)=0$ at the
remaining vertices, interpolates as
\[U(\x)=(\CI u)(\x)=\prod_j(1-\x_j),\qquad \x=(\x_1,\ldots,\x_d)\in
[0,1]^d.\] Such interpolation defines a mapping $\CI:u\mapsto U$
of the space $\CH^1(\Z^d)$ to a certain subspace of piecewise
poly-linear, continuous functions on $\R^d$.
\begin{lem}\label{lem1} The image of the mapping $\CI$ is the
space $\CH\CP^1(\R^d)$ of all piecewise poly-linear, continuous
functions with the finite Dirichlet integral. The quadratic forms
$Q_0[u]$, see \eqref{0:3}, and $D[\CI u]=\int_{\R^d}|\nabla (\CI
u)|^2dx$ are equivalent: there are constants $c,c'>0$ such that
\begin{equation}\label{3:interp.2}
    cQ_0[u]\le D[\CI u]\le c'Q_0[u],\qquad\forall u\in
\CH^1(\Z^d).
\end{equation}
\end{lem}
\begin{proof}
Consider the space $\CL(\CC)$ of poly-linear functions on  the
unit cell $\CC=[0,1]^d$. Clearly, $\dim\CL(\CC)=2^d$. On
$\CL(\CC)$ we consider the quadratic forms
\begin{equation*}
\wt Q[U;\CC]=\sum_{x,y\in\{0,1\}^d\atop x\sim
y}|U(x)-U(y)|^2;\qquad \wt D[U;\CC]=\int_\CC|\nabla U(\x)|^2d\x.
\end{equation*}
These two quadratic forms vanish on the same subspace in
$\CL(\CC)$, consisting of constant functions. Therefore, they are
equivalent, i.e., with some $c,c'>0$ we have
\begin{equation*}
    c\wt Q[U;\CC]\le\wt D[U;\CC]\le c'\wt Q[U;\CC].
\end{equation*}
By adding up similar inequalities for all the cells $\CC+x,\
x\in\Z^d$, we arrive at \eqref{3:interp.2}.

It remains to check that $\CI u\in\CH^1(\R^d)$ for any
$u\in\CH^1(\Z^d)$. It is sufficient to show this for the dense in
$\CH^1(\Z^d)$ subset of functions with finite support. But for any
such $u$ the function $\CI u$ has compact support and, therefore,
it can be approximated in $\CH^1$ by functions in $C_0^\infty$, in
the metric of the Dirichlet integral. Hence, it lies in
$\CH^1(\R^d)$.
\end{proof}
\vs

Now, with any non-negative function $W(x), \ x\in\Z^d$, we
associate a function $\BW=\CJ W$ on $\R^d$, setting $\BW(\x)=W(x)$
for $\x\in \CC+x$ where $\CC$ is the same as above.

\begin{lem}\label{3:interp.W}For any weight $W$ on $\Z^d$ and the
corresponding weight $\BW=\CJ W$ on $\R^d$, one has
\begin{equation}\label{3:interp.3}
     \sum_{x\in\Z^d}W(x)|u(x)|^2\le C_{4.2} \int_{\R^d}\BW(\x)|U(\x)|^2 d\x,
     \qquad U=\CI u,
\end{equation}
for any function $u$ on $\Z^d$. The constant in \eqref{3:interp.3}
depends only on $d$.
\end{lem}
\begin{proof}By linearity, it is sufficient to prove \eqref{3:interp.3}
for $W(x)$ having support at one point, say, at $0\in\Z^d$. Then
the quadratic forms in \eqref{3:interp.3} define Hilbert seminorms
on a $2^d$-dimensional space, with the null subspace for the
quadratic form on the right-hand side contained in the null
subspace for the one on the left-hand side. This gives the
required inequality.
\end{proof}
Note that the estimate inverse to \eqref{3:interp.3} is
impossible. Indeed, the null subspace for the quadratic form on
the left-hand side is strictly wider than for the other one.
\subsection{Hardy type inequalities}\label{hardy}
 We will call a non-negative function
$W(x),$ $x\in\Z^d$ a \emph{Hardy weight} on $\Z^d$,
 if for some
constant $H(W)$ the inequality is satisfied:
\begin{equation}\label{3:Hardy.1}
    \sum_{x\in\Z^d}W(x)|u(x)|^2\le H(W) Q_0[u], \qquad \forall
    u\in\CH^1(\Z^d).
\end{equation}
The best possible constant $H(W)$ in \eqref{3:Hardy.1} will be
called the Hardy constant for $W$. We will say that a Hardy weight
$W$ is {\it normalized}, if $H(W)=1$. This definition carries over
to the lattice case the classical definition of Hardy weights in
$\R^d$:
\begin{equation*}
\int_{\R^d} \BW(\xi)|U(\xi)|^2d\xi\le H(\BW)\int_{\R^d} |\nabla
U|^2d\xi, \qquad \forall U\in \CH^1(\R^d).
\end{equation*}
In the continuous case the complete description of Hardy weights
was found by Maz'ya, see \cite{Maz}, Ch. 8. There, the necessary
and sufficient condition on  a function $\BW(\x)$ to be a Hardy weight in $\R^d$ is
given in the terms of the capacity. By means of the well-known
relation between capacity and measure, a sufficient condition for
$\BW$ to be a Hardy weight can be expressed in more elementary
terms, see, e.g., Proposition 5.1 in \cite{BS2}.
\begin{prop}\label{3:HardyCont} Any function $\BW\ge 0$ in the weak
class $L_{\frac{d}2,w}(\R^d)$, $d\ge 3$, is a Hardy weight on
$\R^d$ with Hardy constant $H(\BW)$ satisfying $H(\BW)\le
C(d)\|\BW\|_{L_{\frac{d}2,w}(\R^d)}$.
\end{prop}
The  discrete analogue of Proposition \ref{3:HardyCont} also
holds.

\begin{prop}\label{3:HardyDisc} Let $W\ge0$ be a function defined on
$\Z^d,\ d\ge3$.  Suppose that $\BW=\CJ W$ is a Hardy weight on
$\R^d$. Then $W$ is a Hardy weight on $\Z^d$, and $H(W)\le C
H(\BW)$. In particular, any $W\in \ell_{\frac{d}{2},w}(\Z^d)$ is a
Hardy weight on $\Z^d$ with $H(W)\le
C(d)\|W\|_{\ell_{\frac{d}{2},w}(\Z^d)}.$
\end{prop}
\begin{proof}  The first statement follows immediately from the
inequalities  \eqref{3:interp.2} and \eqref{3:interp.3}.  As for
the second statement, note that  the functions $W$ on $\Z^d$ and
$\BW=\CJ W$ on $\R^d$ are equimeasurable, therefore $\BW\in
L_{\frac{d}2,w}(\R^d),$ with the same quasi-norm, so we can use
Proposition \ref{3:HardyCont}.
\end{proof}

The simplest example of a function in $\ell_{\frac{d}2,w}(\Z^d)$ is
$W(x)=(|x|^2+1)^{-1}$. By Proposition \ref{3:HardyDisc}, it is a
Hardy weight on $\Z^d$. This justifies the inequality
\eqref{0:Hardy}. \vs

Now we are in a position to prove Proposition \ref{role}. Indeed,
the boundedness of the operator $\BB_V$ is just a re-formulation
of the property of $V$ to be a Hardy weight. Evidently, $\BB_V$ is
compact for any $V$ with finite support. This property extends to
the whole of $\ell_{\frac{d}2,w}^\circ$ by continuity.

\section{Estimates in $\Sg_q,\ q>\frac{d}2$}\label{2q>d}
Having the Hardy type inequalities at our disposal, we now move on
to the estimation of the operator $\BB_V$ in the classes $\Sg_q$
with $q>\frac{d}2$. Results in this section are the direct
analogues (actually, immediate consequences) of the corresponding
results for the operators on $\R^d$, obtained in \cite{BS,BS2}. In
order to distinguish between the discrete and the continuous
Laplacians, we (in  this section only) will denote them by $\D_d$
and by $\D_c$ respectively. \vs

\subsection{Eigenvalue estimates}\label{eigest}
The material of Section \ref{hardy1} allows one to prove that any
statement on the eigenvalue behavior of the operator $\BB_V$,
expressed in terms of the standard, or weak $L_q$-classes,
automatically implies its discrete counterpart. In particular,
this shows that the estimate \eqref{1:rlc-d} follows directly from
\eqref{1:rlc}, thus giving one more proof of \thmref{RLC-d}. Below
we demonstrate, how this idea implements for  obtaining the
discrete analogues of the results in \cite{BS,BS2}.

Let a function $\BW$ be a normalized Hardy weight on $\R^d,\
d\ge3$. The following is the formulation of Theorem 4.1 in
\cite{BS2}.
 For a function $\BF\ge0$ on $\R^d$ we say that
 $\BF\in L_{q,w}(\BW^\frac{d}2)$ if
 \begin{equation*}
    \|\BF\|^q_{L_{q,w}(\BW^\frac{d}2)}=\sup_{t>0}\left(t^q\int
    \limits_{\BF(\x)>t\BW(\x)}\BW^{d/2}d\x\right)<\infty.
 \end{equation*}

\begin{prop}\label{3:stand} Let $d\ge3$, and let $\BW>0$ be a
normalized Hardy weight on $\R^d$. Suppose $\frac{d}{2}<q<\infty$
and $\BV\ge 0$. Then
\begin{equation*}
    \BV/\BW\in L_{q,w}(\BW^{d/2})\ \Rightarrow\
    N_-(-\D_c-\a \BV)\le C_{5.1} \a^q
    \|\BV/\BW\|^q_{L_{q,w}(\BW^{d/2})}
\end{equation*}
where the constant depends on $d$ and $q$.

If, besides,  $\lim\limits_{t\to 0,\infty}\left(t^q\,
\int\limits_{\BV(\x)>t\BW(\x)}
    \BW^{d/2}d\x\right)=0$,
    then $N_-(-\D_c-\a \BV)=o(\a^q).$
\end{prop}

We now can prove the discrete version of this Proposition. Let
$W>0$ be a Hardy weight on $\Z^d,\ d\ge 3$. For a function $F\ge0$
on $\Z^d$, we say that $F\in \ell_{q,w}(W^{d/2})$ if

\begin{equation*}
    \|F\|^q_{\ell_{q,w}(W^{d/2})}=
    \sup_{t>0}\left(t^q\sum\limits_{F(x)>tW(x)}
    W(x)^{d/2}\right)<\infty.
\end{equation*}

\begin{thm}\label{3:EstimDiscrete}Let $d\ge3$, and let $W>0$ be a
function on $\Z^d$, such that $\BW=\CJ W$ is a normalized Hardy
weight on $\R^d$. Suppose $\frac{d}2<q<\infty$ and $V\ge 0$. Then
\begin{equation}\label{3:discr.1}
    V/W\in \ell_{q,w}(W^{d/2})\ \Rightarrow\
    N_-(-\D_d-\a V)\le C_{5.1} \a^q
    \|V/W\|^q_{\ell_{q,w}(W^{d/2})}.
\end{equation}

If, besides, $\lim\limits_{t\to 0,\infty}
\left(t^q\sum\limits_{V(x)>tW(x)}W(x)^{d/2}\right)=0$,
    then $N_-(-\D_d-\a V)=o(\a^q).$
\end{thm}
\begin{proof}
Consider the operator $\BB_\BV$ in the space $\CH^1(\R^d)$,
defined by the quadratic form
\[\bb_\BV[U]=\int \BV(\x)|U(\x)|^2d\x,\qquad \BV=\CJ V.\]
By the Birman-Schwinger principle, see Proposition \ref{BSchw}
(or, more exactly, its continuous version), we have
\begin{equation*}
    N_-(\D_c-\a \BV)=n_+(\a^{-1}, \BB_\BV).
\end{equation*}

The function $\BV$ is equimeasurable with $V$, $\BW$ is
equimeasurable with $W$, and $\BV/\BW$ is equimeasurable with
$V/W$. Therefore the condition $V/W\in \ell_{q,w}(W^{d/2})$
implies $\BV/\BW\in L_{q,w}(\BW^{d/2})$, so we can apply Theorem
\ref{3:stand} and obtain the estimate
\begin{equation}\label{3:disc.3}
    n_+(\a^{-1}, \BB_\BV)\le C_{5.1} \a^q \|V/W\|_{\ell_{q,w}(W^{d/2})}.
\end{equation}

Now we consider the operator $\BB^0_\BV$ defined by the same
quadratic form $\bb_\BV[U]$, but restricted to the space
$\CH\CP^1(\R^d)$ of piecewise  poly-linear functions in
$\CH^1(\R^d)$.  Since we narrowed the domain of the quadratic
form, the eigenvalues of the operator cannot grow, so
\eqref{3:disc.3} leads to
\begin{equation}\label{3:disc.4}
    n_+(\a^{-1}, \BB^0_\BV)\le C_{5.1} \a^q \|V/W\|_{\ell_{q,w}(W^{d/2})}.
\end{equation}

 By  \eqref{3:interp.2} and \eqref{3:interp.3}, the quadratic
 form $\bb_V[u]$ in \eqref{1:1} is estimated from above by
 $\bb_\BV [\CI u]$, and the quadratic form
 $Q_0[u]$ in \eqref{0:3} is estimated from below by
 $\int|\nabla(\CI u) |^2 d\xi$.
 Therefore the eigenvalues of the operator $\BB_V$ are majorized by
 the eigenvalues of the operator $\BB_\BV$, and the estimate
 \eqref{3:discr.1} follows from \eqref{3:disc.4}. The second part
 of Theorem \ref{3:EstimDiscrete} is proved by a similar reasoning.
 \end{proof}

Theorem \ref{3:EstimDiscrete} gives much freedom in choosing the
Hardy weight $W$. The most standard choice is
$W(x)=c(|x|^2+1)^{-1}$ where $c>0$ is the normalizing constant.
For this $W$, the corresponding particular case of
\thmref{3:EstimDiscrete} was presented in \cite{RS} as Theorem
9.2. Below we repeat its formulation, in order to correct a
misprint in \cite{RS}.

\begin{cor}\label{partic}
Let $d\ge3$ and $2q>d$, and let $V\ge 0$. Then
\begin{equation}\label{estim}
    N_-(\BH_{\a V})\le C(d,q)\a^q\sup_{t>0}\left(t^q
    \sum\limits_{(|x|^2+1)V(x)>t}(|x|^2+1)^{-\frac{d}2}\right).
\end{equation}
\end{cor}

\subsection{An example.}\label{ex} Here we analyze an example which
shows that the estimate \eqref{estim} is order-sharp for certain
potentials $V$. Note that in the setting of Section \ref{2q<d}
such examples are unnecessary, since its results give two-sided
estimates of the quasi-norm $\|\BB_V\|_{\Sg_q}$ for $2q<d$.

What we give below, is an analogue of Example 4.1 in \cite{BS}.
However, in the discrete case we were  not able to establish the
asymptotic behavior of the eigenvalues $\l_j(\BB_V)$, and we only
give for them some lower estimate.

In our example the potential, for $|x|>1$, has the form
\begin{equation*}
    V(x)=|x|^{-2}(\log(|x|))^{-\frac{1}{q}}, \qquad 2q>d.
\end{equation*}
By Corollary \ref{partic}, for the corresponding Birman-Schwinger
operator $\BB_V$, the upper eigenvalue estimate holds,
\begin{equation*}
    n_+(s,\BB_V)\le C s^{-q}.
\end{equation*}
We will show that this estimate cannot be improved. More
specifically, we will prove that $\BB_V\not\in \Sg_q^\circ$. To
this end, we will construct an orthonormal in $\CH^1$ sequence of
functions $\varf_n$ such that
\begin{equation}\label{3:Ex.6a}
    \bb_V[\varf_n] \ge cn^{-\frac1q},\qquad c>0.
\end{equation}
Then the desired lower estimate will follow from the well known
results on compact operators in a Hilbert space. Namely, it
follows from Lemma II.4.1 in \cite{GK}, or Theorem 11.5.7 in
\cite{BSbook}, that
\[ \sum_{n=1}^m \l_n(\BB_V)\ge \sum_{n=1}^m\bb_V[ \varf_n],
\qquad \forall m\in\N.\]
 Since $\frac1q<\frac2d<1$, this inequality and \eqref{3:Ex.6a} imply
\[\sum_{n=1}^m \l_n(\BB_V)\ge cm^{1-\frac1q},\]
which is inconsistent with $\BB_V$ being in the class
$\Sg_q^\circ.$\vs

So, it remains to construct the functions $\varf_n$. We obtain
them by normalizing a certain sequence $\{u_n\}$. For constructing
the latter, we use the Fourier representation of functions on the
lattice. With a function $u\in\CH^1$ we associate its Fourier
transform,
\[\wh{u}(z)=(2\pi)^{-d/2}\sum_{x\in\Z^d}u(x)e^{-ixz},\qquad z\in
\Tor^d, \] where $\Tor^d$ stands for the $d$-dimensional torus.
Conversely,
\[u(x)=(2\pi)^{-d/2}\int_{\Tor^d}\wh{u}(z)e^{ixz}dz.\]
Then
\begin{equation}\label{3:Ex.6b}
    Q_0[u]=\int_{\Tor^d}|\wh u(z)|^2\om(z)dz,\qquad
 \om(z)=4\sum\limits_{j=1}^d\sin^2(z_j/2).
\end{equation}
We take $h_n=4^{-n}$ and set \begin{equation*}v_n(z)=\wh{u}_n(z)
=\chi\left(\frac{z_1}{h_n}-3\right)\prod_{j=2}^d\chi\left(\frac{z_j}{h_n}
\right),
\end{equation*}
where $\chi$ is the characteristic function of the interval
$(-1,1)$. So, the function $v_n(z)$ has support in the square with
the side length $2h_n$ and with the center moved by $3h_n$ in the
direction of $z_1$ axis. Obviously, the supports are disjoint, so
the functions $v_n$ are orthogonal in $L_{2,\om}$. The latter
notation stands for the $L_2$-space with the weight $\om$. Thus,
the functions $u_n$ are mutually orthogonal in $\CH^1$.\vs

The functions $u_n$ can be calculated explicitly:
\begin{equation}\label{expl}
    u_n(x)=(2\pi)^{-d/2}e^{3ih_nx_1}\prod\limits_{j=1}^d
    \frac{2\sin(h_nx_j)}{x_j}.
\end{equation}
We have $\|u_n\|^2_{\CH^1}=C\|v_n\|^2_{L_{2,\omega}}$. Since
$\omega(z)\asymp h_n^2$ on the support of $u_n$, this gives
$\|v_n\|^2_{L_{2,\omega}}\asymp h_n^2 \|v_n\|^2_{L_2}\asymp
h_n^{2+d}$.

Now, from \eqref{expl} we derive that
\begin{equation*}
  \bb_V[u_n]\ge c  \sum_{x\in \Z^d} V(x)\prod\limits_{j=1}^d
  \frac{4\sin^2(h_nx_j)}{|x_j|^2}.
\end{equation*}
We need  to estimate the last expression from below. To this end,
we restrict summation in the last sum to the region
$\Om_n\subset\Z^d$ where all $|h_nx_j|<\pi/2$ and $|x|>4$. In
$\Om_n$, we have $\sin^2(h_nx_j)\ge
\left(\frac{2}{\pi}h_nx_j\right)^2$. Therefore,
\begin{equation*}
\bb_V[u_n]\ge ch_n^{2d}\sum_{x\in \Om_n}V(x)=c
h_n^{2d}\sum_{4<|x|\le c h_n^{-1}}|x|^{-2}
   (\log|x|)^{-\frac1q}.
\end{equation*}

A lower bound is given by the integral
\begin{equation*}
   h_n^{2d} \int_4^{ch_n^{-1}}r^{-3+d}(\log r)^{-\frac1q}dr
   \asymp h_n^{-2+d}|\log h_n| ^{-\frac1q}.
\end{equation*}
Therefore,
\begin{equation*}
    \bb_V[u_n]/\|u_n\|_{\CH^1}^2\ge c |\log h_n|^{-\frac1q}\ge
     c n^{-\frac1q}.
\end{equation*}
 So, by normalizing the functions $u_n$ we obtain the sequence
 $\{\varf_n\}$ satisfying \eqref{3:Ex.6a}.\vs

It is interesting to notice that the test functions guaranteeing
the lower estimate for $2q<d$ were constructed as having disjoint
supports, while for $2q>d$ such functions have disjoint supports
of their Fourier transform.

\section{Sparse potentials}\label{sparse}
The results of Sections \ref{2q<d}, \ref{2q>d} allow one to
construct, for any prescribed value of $q\neq\frac{d}2$, the
potentials $V$ such that $\BB_V\in\Sg_q$, but
$\BB_V\notin\Sg_q^\circ$. For the borderline value $q=\frac{d}2$,
\thmref{RLC-d} leaves open the question of existence of such
potentials.

 In this section we consider a special class of potentials for which the
theory can be advanced much further. In particular, we answer the
above question by showing that for any number sequence
$p_j\searrow 0$ such that $p_{j+1}/p_j\to 1$, a potential $V$ does
exist, such that the sequence of eigenvalues $\l_j(\BB_{V})$
asymptotically behaves as $\{p_j\}$.

\subsection{Green function of the discrete
Laplacian.}\label{greendiscr} The operator $(-\D)^{-1}$ acts as a
discrete convolution, its kernel can be represented by the explicit
formula:
\[ h_y(x)=h_0(x-y),\qquad \forall y\in\Z^d,\]
 where
\begin{equation}\label{green}
    h_0(x)=(2\pi)^{-d}\int_{\Tor^d}\frac{e^{ixz}}
    {4\sum_{j=1}^d\sin^2(z_j/2)}dz.
\end{equation}
 Note that here the denominator coincides with the weight
function $\om$ in \eqref{3:Ex.6b}.

 We will call $h_0$ the Green function. One should be
careful when studying its properties, since the point $\l=0$ lies
in the spectrum of $-\D$. However, all the difficulties can be
easily overcome by systematic use of the representation
\eqref{green}. See, in particular, \cite{D}, where the case $d=3$
is analyzed. For any $d>3$ the reasoning is similar.\vs

The function $h_0(x)$ lies in $\CH^1(\Z^d)$, is harmonic outside
the point $x=0$, and its value at this point is
\[ h_0(0)=\mu^2=(2\pi)^{-d}\int_{\Tor^d}\frac{dz}
    {4\sum_{j=1}^d\sin^2(z_j/2)}.\]
Then also $h_y(y)=h_0(0)=\mu^2$ for any $y\in\Z^d$. Besides,
\[(\D h_y)(y)=-1,\qquad\forall y\in\Z^d.\]

Let $u$ be a function with finite support. Then summation by parts
leads to the equality
\begin{equation}\label{green2}
    (u,h_y)=-\sum_{x\in\Z^d}u(x)(\D h_y)(x)=u(y).
\end{equation}
Starting from \eqref{green2}, the scalar products, and also the
norms, are taken in $\CH^1$. Since finitely supported functions
are dense in $\CH^1$, the equality \eqref{green2} extends by
continuity to all $u\in\CH^1$. In particular,
\begin{equation}\label{almort}
    (h_{y},h_{y_1})=h_{y}(y_1)=h_0(y-y_1).
\end{equation}
Taking in \eqref{almort} $y_1=y$, we find that
\begin{equation*}
    \|h_y\|=\mu,\qquad \forall y\in\Z^d.
\end{equation*}

We have $0< h_0(x)\le \mu^2$ for all $x\in\Z^d$. It also follows
from the representation \eqref{green} that
\begin{equation}\label{decay}
    h_0(x)\le C_{6.4}|x|^{-(d-2)},\qquad x\neq0,
\end{equation}
with a constant depending only on $d$. The relations
\eqref{almort} and \eqref{decay} show that for the points
$y,y_1\in\Z^d$ lying far enough from each other, the functions
$h_y,h_{y_1}$ are `almost orthogonal'. It is convenient to
normalize these functions, so that further on we work with
\begin{equation*}
    \wt h_y=\mu^{-1}h_y.
\end{equation*}

\subsection{Sparse subsets in $\Z^d$.}\label{sparsub} Let $Y$ be a
subset in $\Z^d,\ d\ge3$, and let $\CH_Y^1$ stand for the subspace
in $\CH^1(\Z^d)$ spanned by the functions $h_y,\ y\in Y$. We say
that the set $Y$ is {\it sparse} (or, in more detail, {\it
strongly sparse}), if in $\CH^1(\Z^d)$ there exists a compact
operator $\BT$, such that the operator $\BI-\BT$ has bounded
inverse and the functions
\begin{equation}\label{spar}
    e_y=(\BI-\BT)^{-1}\wt h_y,\qquad y\in Y,
\end{equation}
form an orthonormal system in $\CH^1$.

We say that $Y$ is {\it weakly sparse}, if in the above definition
we replace the requirement of $\BT$ being compact by its
boundedness.\vs

We are going to describe (in Lemma \ref{sparlem} below) a rather
general way to construct sparse subsets. To this end, we need some
preliminary material.

Suppose that a sequence $\{\psi_j\},\ j\in\N,$ of elements of a
separable Hilbert space $\GH$ can be represented as
\[ \psi_j=(\BI-\BT)\varf_j,\qquad j\in\N,\]
where $\{\varf_j\}$ is an orthonormal basis in $\GH$ and $\BT$ is
a linear operator of the class $\GS_q$ with some $q,\
0<q\le\infty$. Then we say that $\{\psi_j\}$ is a $q$-basis. We
recall that $\GS_\infty$ stands for the space of all compact
linear operators in $\GH$, and $\GS_q,\ 0<q<\infty$, stands for
the Schatten ideal; see, e.g., \cite{BSbook,GK}.

The following  result is due to Prigorskii, see Theorem 5 in
\cite{Pr}.
\begin{prop}\label{pri}
Let $\{\psi_j\},\ j\in\N,$ be a complete and $\om$-linearly
independent sequence in $\GH$. It is a $q$-basis if and only if
the matrix
\[\begin{pmatrix}(\psi_j,\psi_k)_\GH-\d_{jk}\end{pmatrix}\]
belongs to the class $\GS_q$.
\end{prop}

Here the property of $\om$-linear independence of the system
$\{\psi_j\}$  means
that the assumptions
 \[\bsymb\y=\{\y_j\}\in \ell_2,\qquad \sum_j \y_j\psi_j
=0\] yield $\bsymb\y=0$, see definition in \cite{GK}, Section VI.2.4.

Note that for $q=2$ the statement of Proposition \ref{pri} turns
into that of Theorem VI.3.3 in \cite{GK}. The proof for the
general case follows the same scheme as in \cite{GK}.\vs

Now we introduce the quantities that appear in the formulation of
Lemma \ref{sparlem} below. Let $Y$ be a subset of $\Z^d$.
 Given a $y\in Y$, we denote
\[r_y=\dist(y,Y\setminus\{y\})\]
 and
 \[ [y]=\#\{x\in Y:|x|\le|y|\}.\]
 \begin{lem}\label{sparlem}
 Suppose $d\ge3$, and let $Y\subset\Z^d$ be a set, such that
  $[y]r_y^{-(d-2)}\to 0$ as
 $y\in Y,\ |y|\to\infty$.
 Suppose also that
\[\sum_{y\in Y}r_y^{-(d-2)}\le A,\qquad
\sup_{y\in Y}\left([y]r_y^{-(d-2)}\right)\le A, \qquad\forall y\in
Y,\]
 with a constant $A$ satisfying
\begin{equation}\label{sparlem.1}
    2AC_{6.4}<\mu^2.
\end{equation}
Then the set $Y$ is strongly sparse.
\end{lem}

\begin{proof}
Consider the Gram matrix
\[  G=\begin{pmatrix}(\wt h_x,\wt h_y)\end{pmatrix}_{x,y\in Y}.\]
Its diagonal elements are equal to one, and by \eqref{almort}, for
every $y\in Y$ we have
\begin{eqnarray}
    \sum_{x\in Y,x\neq y}|(\wt h_x,\wt h_y)|\le&
    C_{6.4}\mu^{-2}\left(\sum\limits_{x\in Y\setminus\{y\},|x|\le|y|}
    +\sum\limits_{x\in Y,|x|>|y|}\right)|x-y|^{-(d-2)}\notag\\
    \le& C_{6.4}\mu^{-2}\left([y]r_y^{-(d-2)}+\sum\limits_{x\in Y,|x|>
|y|}r_x^{-(d-2)}\right).\label{sparlem.2}
\end{eqnarray}

Since the matrix $G$ is Hermitian, the same inequality holds if
the roles of $x$ and $y$ are interchanged. By the Shur test (see,
e.g., Theorem 2.5.6 in \cite{BSbook}), this implies that the
matrix $G-I$ defines a bounded operator in $\ell_2(Y)$ and,
moreover, \eqref{sparlem.1} yields $\|G-I\|<1$.

In order to prove that $G-I\in\GS_\infty$, we show that the
`truncated' matrices
\[G_N=\begin{pmatrix}(\wt h_x,\wt h_y)\end{pmatrix}_{x,y\in Y;\
|x|,|y|\le N}\] converge to $G$ in the operator norm. For
estimating $\|G-G_N\|$ we again use the Shur test. For the rows
with $|y|>N$ the estimate \eqref{sparlem.2} survives, and under
the assumptions of Lemma the sums on the right are small if $N$ is
large enough. If $|y|\le N$, then this sum should be replaced by
$\sum\limits_{x\in Y,|x|> N}r_x^{-(d-2)}$ which is also small. So,
$G-I\in\GS_\infty$.

Next, we check that the system $\{\wt h_{y}\}_{y\in Y}$ is
$\om$-linearly independent. Indeed, by taking the scalar products
with $\wt h_x,\ x\in Y,$ we come to the infinite system of linear
homogeneous equations, $G\bsymb\y=0$. Since $\|G-I\|<1$, the
operator $G$ is invertible in $\ell_2$, so that the only
$\ell_2$-solution of this system is trivial: $\bsymb\y=0$, and we
are done.

Now Proposition \ref{pri} applies, with $q=\infty$. So, we get an
operator $\BT_0\in\GS_\infty$ in the space $\CH^1_Y$, such that
the system $\{(\BI-\BT_0)\wt h_y\}_{y\in Y}$ is orthonormal.
Extending $\BT_0$ by zero to the orthogonal complement of
$\CH^1_Y$, we obtain the operator $\BT$ that satisfies all the
properties we need.
\end{proof}

We do not have a method that would allow us to construct subsets
that are weakly sparse but not sparse. Still, we consider the
notion of weak sparseness useful, since it gives us a way to see
the difference between implications of both types of sparseness.

\subsection{Sparse potentials.}\label{sparp}
We say that $V\ge0$ is a sparse (a weakly sparse) potential on
$\Z^d$, if its support
\[ Y_V :=\{x\in\Z^d: V(x)>0\}\]
is a sparse (respectively, weakly sparse) subset.\vs

Let $V$ be a weakly sparse potential. Then, using the equalities
\eqref{green2} and \eqref{spar}, we see that the quadratic form
$\bb_V[u]$ can be written as
\[ \bb_V[u]=\sum_{y\in Y_V}V(y)|(u,h_y)|^2=
\mu^2\sum_{y\in Y_V}V(y)|(u,(\BI-\BT)e_y)|^2.\]
 Along with the
operator $\BB_V$ generated by this quadratic form, consider also
the self-adjoint in $\CH^1(\Z^d)$ operator
\begin{equation}\label{oper}
    \BN_V=\mu\sum_{y\in Y_V} \sqrt{V(y)}(\cdot,e_y)e_y.
\end{equation}
We have
\[\BN_V(\BI-\BT^*)=\mu\sum_{y\in Y_V} \sqrt{V(y)}(\cdot,\wt h_y)e_y,\]
whence
\[ \|\BN_V(\BI-\BT^*)u\|^2=\mu^2\sum_{y\in Y_V}V(y)
|(u,\wt h_y)|^2=\bb_V[u].\]
 This means that
\begin{equation}\label{bv}
    \BB_V=(\BI-\BT)\BN_V^2(\BI-\BT^*).
\end{equation}

This representation allows one to obtain the following simple
result that is not covered by Proposition \ref{3:HardyDisc}.

\begin{thm}\label{hardy-sparse}
Let $V\ge 0$ be a function on $\Z^d,\ d\ge 3$, such that its
support $Y_V$ is weakly sparse. Then $V$ is a discrete Hardy
weight if and only if $V$ is bounded. Moreover, the following
two-sided inequality is satisfied for the Hardy constant $H(V)$:
\begin{equation}\label{harspconst}
    C\|V\|_{\ell_\infty}\le H(V)\le C'\|V\|_{\ell_\infty}.
\end{equation}
where $C=\|(\BI-\BT)^{-1}\|^{-2}\mu^2$ and
$C'=\|\BI-\BT\|^2\mu^2$.
\end{thm}
\begin{proof}
In the equivalent terms, we have to find the boundedness
conditions of the operator $\BB_V$. They immediately follow from
the representation \eqref{bv} and the equality
$\|\BN_V\|=\mu\|V\|_{\ell_\infty}^{1/2}.$
\end{proof}

The next result is also  a consequence of \eqref{bv}. Recall that
for a function $V$, such that $V(x)\to0$ as $|x|\to\infty$, we
write $V_j^*, j\in\N,$ for the numbers $|V(x)|$ rearranged in the
non-increasing order, see Section \ref{B-Sch}.
\begin{thm}\label{hardy-sparse1}
Let $V\ge 0$ be a function on $\Z^d,\ d\ge 3$, such that its
support $Y_V$ is weakly sparse. Then the operator $\BB_V$ is
compact if and only if $V(x)\to 0$ as $|x|\to \infty$. Moreover,
the following two-sided inequality is satisfied for the
eigenvalues $\l_j(\BB_V)$:
\begin{equation*}
      CV_j^*\le \l_j(\BB_V)\le C'V_j^*,\qquad\forall j\in\N,
    \end{equation*}
    with the same constants as in \eqref{harspconst}.
\end{thm}
The proof is the same as for \thmref{hardy-sparse}.

\vs

The next result is more advanced, and it requires the potential to
be sparse but not weakly sparse. We derive this result for
potentials $V$ which meet an additional condition: we assume that
the corresponding sequence $\{V_j^*\}$ is {\it moderately
varying}. We say that an (infinite) sequence $\{p_j\}$ of positive
numbers is moderately varying, if $p_j\searrow 0$ and
$p_{j+1}/p_j\to1$. We shall use a result of M.G. Krein, see
Theorem 5.11.3 in \cite{GK}. Below we reproduce its formulation,
restricting ourselves to the situation we need.
\begin{prop}\label{krein} Let $\BH\ge0$ and
$\BS$ be self-adjoint, compact operators. Suppose
$\rank\BH=\infty$ and the sequence of non-zero eigenvalues
$\l_j(\BH)$ is moderately varying. Then for the operator
$\BM=\BH(\BI+\BS)\BH$ one has
\begin{equation*}
    \lim_{j\to\infty}\frac{\l_j(\BM)}{\l_j^2(\BH)}=1.
\end{equation*}
\end{prop}

The next theorem is the main result of this section.
\begin{thm}\label{sppot}
Let $V\ge0$ be a sparse potential, such that the numbers $V_j^*$
form a moderately varying sequence. Then
\begin{equation*}
    \lim_{j\to\infty}\frac{\l_j(\BM)}{V_j^*}=1.
\end{equation*}
\end{thm}
\begin{proof}
  The non-zero spectrum of the operator \eqref{bv} coincides
with that of the operator
\begin{equation}\label{kr}
    \BM_V:=\BN_V(\BI-\BT^*)(\BI-\BT)\BN_V=\BN_V(\BI+\BS)\BN_V,
\end{equation}
where
\begin{equation*}
    \BS=-\BT-\BT^*+\BT^*\BT.
\end{equation*}

Now we apply Proposition \ref{krein} to the operators $\BH=\BN_V$
and $\BM_V$, given by \eqref{oper} and \eqref{kr} respectively.
All the assumptions of proposition are evidently satisfied, and we
get the desired result.
\end{proof}

\begin{cor}\label{ccor}
For any moderately varying sequence $\{p_j\}$ there exists a sparse
potential $V\ge 0$, such that $\l_j(\BB_V)\sim p_j$.
\end{cor}

Taking here $p_j=j^{-q}$ with an arbitrary $q>0$, we obtain a
potential $V$ such that the eigenvalues $\l_j(\BB_V)$ behave as
$p_j$ and, therefore,
\[ N_-(\BH_{\a V})=O(\a^q);\qquad N_-(\BH_{\a V})\neq o(\a^q).\]
For $q=\frac{d}2$, this solves the problem stated in the end of
Introduction. This corollary allows one also to construct
potentials $V$, such that the eigenvalues $\l_j(\BB_V)$ decay
arbitrarily slowly, say, $\l_j(\BB_V)\sim (\log j)^{-1}$. This
corresponds to the exponential growth of the function $N_-(\BH_{\a
V})$. Applying \thmref{hardy-sparse1}, we easily obtain also
potentials with the logarithmic growth of this function.

\section{Discussion and concluding remarks}\label{concl}
\subsection{}\label{concl0} Let us briefly summarize our main
results.

1. Each estimate of the type
\begin{equation}\label{7:0}
    N_-(\BH_{\a V})=O(\a^q)
\end{equation}
with  $q\ge \frac{d}2$, known for the continuous Hamiltonian
\eqref{0:1}, has an analogue for its discrete counterpart. What
is more, such  analogue can be derived directly from its
prototype.

2. For $q\ge\frac{d}2$ there are also estimates \eqref{7:0} of a
different origin, that hold for the discrete Hamiltonian
\eqref{0:1} but do not have continuous prototypes. These are
estimates for the sparse potentials. Unlike the results mentioned
above, these ones give also lower estimates of the same order.

3. Estimates of the type \eqref{7:0} with $q<\frac{d}2$, that hold
for the discrete Hamiltonian, have no continuous analogues either.
They give necessary and sufficient condition on the potential, in
order that $\BB_V\in\Sg_q$ with a given $q$ to be satisfied.

\subsection{}\label{concl1a} Here we discuss the case of
non-sign-definite potentials $V\to 0$. As usual, we denote
\[ V_\pm=\frac12(|V|\pm V).\]
For the operator $\BH_{\a V}$ with such potential, the eigenvalues
may appear also to the right of the point
$2d=\max\s_{\ess}(\BH_{\a V})$, and along with $N_-(\BH_{\a V})$
one should consider the function
 $N_+(2d;\BH_{\a V})$ defined as the number of  such eigenvalues,
 counted with multiplicities.\vs

 Consider the mapping
 \[ \G:u(x)\mapsto (-1)^{x_1+\ldots+ x_d}u(x).\]
 It defines a unitary operator  in $\ell_2(\Z^d)$, and it is easy to see that
 \begin{equation*}
    \G(-\D-\a V)=(\D+2d\BI-\a V)\G,
\end{equation*}
therefore,
\begin{equation}\label{7:x}
    N_+(2d;\BH_{\a V})=N_-(\BH_{-\a V}).
\end{equation}

The inequalities
\begin{equation}\label{nonsign}
    N_-(\BH_{\a V})\le N_-(\BH_{\a V_+}),\qquad
N_+(2d;\BH_{\a V})\le N_-(\BH_{\a V_-})
\end{equation}
reduce the problem of estimating the functions on the left to the
similar problem for the non-negative $V$. The first
inequality in \eqref{nonsign} follows directly from the variational principle. To
derive the second, one uses also the equality \eqref{7:x}.

 Sometimes,  lower estimates of the same
 type are also possible. In particular, this is the case for the
 results in Section \ref{2q<d}.

\subsection{} Up to now, the problem of counting the
number of negative eigenvalues of the multidimensional discrete
\sh operator did not attract much attention of the specialists.
The only paper we are aware of, is \cite{kar}, where the operator
(1.1) with $\a=1$ was considered, under the assumptions that can
be formulated as follows:
 \[V=V_1-V_2,\qquad V_1,\ V_2\ge0,\] and,
moreover,
\[V_2\ge\nu_0>0;\qquad V_1\in\ell_r,\ r>1.\]
The main result consists in the estimate
\begin{equation}\label{kar}
    N_-(\BH_V)\le(e/\nu_0)^r\sum_{x\in\Z^d}V_1(x)^r
\end{equation}
that holds for all $d\ge1$, rather than only for $d\ge3$. In the
latter case, the estimate \eqref{kar} is much weaker than any of
our estimates.

Below we analyze the estimate \eqref{kar}. The conditions on $V$
imply, for each $n\in\Z^d$, that $V(n)\le V_1(n)-\nu_0$. Hence,
$V_+(n)>0$ yields $V_1(n)>\nu_0$. The number $\CN$ of such indexes
$n$ is always finite, and moreover,
\[ \CN\le\nu_0^{-r}\sum_{x\in\Z^d}V_1(x)^r.\]
Using now \eqref{nonsign} and the trivial inequality \eqref{0:5},
we come to an estimate similar to \eqref{kar}, but without the
redundant factor $e$.

In other words, the estimate \eqref{kar} needs nothing for its
validity, except for the non-negativity of the operator $-\D$. In
particular, this explains why it holds in any dimension.\vs
\subsection{}\label{concl1}
A problem, closely related to the one of counting eigenvalues,
concerns the Lieb -- Thirring type inequalities. This problem was
considered in \cite{HS}. Let $E_j^-(\BH_V), \
E_j^+(\BH_V)$ stand for the eigenvalues of the operator $\BH_V$,
lying on $(-\infty,0)$ and on $(2d,\infty)$ respectively. The
following estimates were established in \cite{HS} for any $d\ge1$,
along with other interesting results: if $V\in\ell_q(\Z^d)$ for
some $q\ge1$, then
  \[  \sum_j|E_j^-(\BH_V)|^q+\sum_j(E_j^+(\BH_V)-2d)^q
    \le\sum_{x\in\Z^d}|V(x)|^q;\]
  if $V\in\ell_{q+\frac{d}2}(\Z^d)$, again for some $q\ge1$, then
\[\sum_j|E_j^-(\BH_V)|^q+\sum_j(E_j^+(\BH_V)-2d)^q
    \le C(d,q)\sum_{x\in\Z^d}|V(x)|^{q+\frac{d}2},\]
with some, explicitly given constant factor.

In this respect we note that from our results it easily follows
that in dimensions $d\ge3$ a similar inequality holds for any
$q>0$, but with an additional and unspecified factor $C'(d,q)$ on
the right. Moreover, in the second estimate the assumption
$V\in\ell_{q+\frac{d}2}(\Z^d)$ can be replaced by
$V\in\ell_{q+r}(\Z^d)$ with an arbitrary $r\in(0,\frac{d}2]$, and
with the corresponding change in the right-hand side.

In this paper we do not look for the optimal constants in such
estimates. Note that traditionally just seeking the optimal
constants is the main issue when dealing with the Lieb -- Thirring
inequalities. So, in this respect our results do not give anything
new.\vs

\subsection{}\label{subs2} Our initial goal when starting the work
on this paper was to construct potentials on $\Z^d,\ d\ge3$, such
that $N_-(\BH_{\a V})=O(\a^\frac{d}2)$ but $N_-(\BH_{\a V})\neq
o(\a^\frac{d}2)$. In \cite{RS} this problem was mentioned as
unsolved. We solved   it here (in Section \ref{sparse}) by using the
sparse potentials.

Sparse potentials were already used in the spectral theory of the
 multidimensional Schr\"odinger operator in the papers \cite{MV,MV1}.
 The authors studied the continuous spectrum of such operators,
 including the scattering for the pair $(\BH_0,\BH_V)$.  Their
 definition of sparseness always requires $\#\left\{y\in
Y_V:|y|<r\right\}=o(r^d)$.
 Further requirements  depend on the problem
 considered. The property \eqref{almort} of
almost orthogonality, which is the basis for our analysis, does
not appear in their approach.

 Note that the use of sparse potentials in the theory of
 one-dimensional operators has a long history, starting from the seminal
 work \cite{Pear} by Pearson.

\subsection{}\label{d1} The spectrum (including its discrete component) of the
Hamiltonian of the form \eqref{0:1} in dimension one is well
studied, mostly due to the possibility to use the theory of Jacobi
matrices. The first results on counting the number of eigenvalues
outside the interval $[0,2]$ are due to Geronimo \cite{Ger2}, \cite{Ger}, see
also \cite{HS}, where a minor error in \cite{Ger} was
corrected. Other results in \cite{HS} concern the Lieb --
Thirring type inequalities, similar to those for $d\ge 3$
described in Section \ref{concl1} but much more elaborated.

\subsection{}\label{d2} We do not touch upon the case $d=2$ in this paper.  Some facts can be derived from the corresponding
results for the continuous case, by means of the technique
developed in Section \ref{hardy1}. However, the results for $d=2$ are incomplete, as well
as in the continuous version of the theory, see the survey
\cite{RS}, and they leave many natural
questions unanswered. In particular, the approach based upon  sparse
potentials does not work for $d=2$, so that in this case we still
have no examples giving $N(\BH_{\a V})=O(\a)$ but $N(\BH_{\a
V})\neq o(\a)$.

\end{document}